\documentclass{amsart}
\usepackage{amsthm}
\usepackage[utf8]{inputenc}
\usepackage{appendix}
\usepackage{cite}
\usepackage{amssymb,amsmath,amsthm}
\usepackage{hyperref}

\newcommand{\amsprimary}[1]{{\footnotesize\noindent AMS 2010 \textit{Mathematics subject
classification:} Primary #1\vspace{1pc}}}
\newcommand{\keywordsnames}[1]{{\footnotesize\noindent\textit{Key words:} #1\vspace{1pc}}}

\newtheorem{theorem}{Theorem}
\newtheorem{teo}{Theorem}
\newtheorem{prop}[teo]{Proposition}

\newtheorem{lemma}[teo]{Lemma}

\theoremstyle{definition}

\theoremstyle{remark}

\title[Logarithmic diffusion]
{On the one dimensional Logarithmic diffusion equation with nonlinear Robin boundary conditions}
\author{Jean Cortissoz}
\address[Jean Cortissoz]{Universidad de los Andes, Bogot\'a DC, COLOMBIA}
\email[Corresponding author]{jcortiss@uniandes.edu.co}
\author{C\'esar Reyes}
\address[C\'esar Reyes]{Universidad Manuela Beltr\'an, Bogot\'a DC, COLOMBIA}
\email{ca.reyes10@uniandes.edu.co}
\date{}

\begin{document}

\maketitle

\begin{abstract}
In this paper we investigate the one dimensional (1D) logarithmic diffusion
equation with nonlinear Robin
boundary conditions, namely,
\[
\left\{
\begin{array}{l}
\partial_t u=\partial_{xx} \log u\quad
\mbox{in}\quad \left[-l,l\right]\times \left(0, \infty\right)\\
\displaystyle
\partial_x u\left(\pm l, t\right)=\pm 2\gamma u^{p}\left(\pm l, t\right),
\end{array}
\right.
\]
where $\gamma$ is a constant. Let $u_0>0$ be a smooth function
defined on $\left[-l,l\right]$, and which 
satisfies the compatibility condition
$$\partial_x \log u_0\left(\pm l\right)= \pm 2\gamma u_0^{p-1}\left(\pm l\right).$$
We show that for $\gamma > 0$, $p\leq \frac{3}{2}$ 
solutions to the logarithmic diffusion equation above with initial data $u_0$
are global and blow-up in infinite time, and for $p>2$ there is finite time blow-up.
Also, we show that in the case of $\gamma<0$, $p\geq \frac{3}{2}$,
solutions to the logarithmic diffusion equation with initial data $u_0$
are global
and blow-down in infinite time,
but if $p\leq 1$ there is finite time blow-down. For
some of the cases mentioned above, and some particular
families of examples, we provide 
blow-up and blow-down rates. Our approach is partly based on studying
the Ricci flow on a cylinder endowed 
with a $\mathbb{S}^1$-symmetric metric. Then, we bring
our ideas full circle by proving a new long time
existence result for the Ricci flow on a cylinder without
any symmetry assumption. Finally, we show a blow-down
result for the logarithmic diffusion equation on a disc.
    
\end{abstract}
 
{\keywordsnames {Logarithmic diffusion; nonlinear Robin boundary conditions; Ricci flow on surfaces.}}

{\amsprimary {53C44; 35K55; 35K57; 58J35.}}

\section{Introduction}

\subsection{}
In this paper we are concerned with the study of the logarithmic diffusion equation with nonlinear 
Robin boundary condition, namely
\begin{equation}
\label{eq:boundaryvaluelogarithmic}
\left\{
\begin{array}{l}
\partial_t u= \partial_{xx} \log u\quad
\mbox{in}\quad \left[-l,l\right]\times \left(0, \infty\right)\\
\displaystyle
\partial_x u\left(\pm l, t\right)=\pm 2\gamma u^{p}\left(\pm l, t\right),
\end{array}
\right.
\end{equation}
with $\gamma$ a constant.
Here we have employed the following notation:
\[
\partial_x:=\frac{\partial}{\partial x},\quad \partial_{xx}:=\frac{\partial^2}{\partial x^2},
\quad \partial_t:=\frac{\partial}{\partial t}.
\]
The logarithmic diffusion equation
appears naturally in physics when studying the behaviour of
a thermalized electronic cloud (that is, its density 
function satisfies Maxwell's distribution) as is shown in
\cite{longman76}, and it 
can be interpreted as a limiting case of the family
of porous media equations (see \cite{vazquez07}). In Geometric Analysis, the logarithmic diffusion equation appears
in the study of the Ricci flow in surfaces, in particular in $\mathbb{R}^2$
(see for instance \cite{daskalopoulos04,daskalopoulos06}). Also,
the logarithmic diffusion equation with Dirichlet boundary conditions, 
and with a source have received some attention,
and in each of these instances have revealed some interesting behaviour
(see \cite{shimojo18} and the references therein).
Our aim is to start a more in depth study on the
behaviour of the logarithmic diffusion equation with Robin nonlinear
boundary conditions, 
and our motivation comes from its relation with the Ricci flow
on surfaces with boundary. 

Henceforth, we adopt the
following notation:
The expression $a\lesssim b$ means that there is a constant $C>0$ such that $a\leq Cb$.
$a\sim b$ means that there is a constant $C>0$ such that $C^{-1}a \leq b \leq Ca$.
For $u:\Omega\times \left[0,T\right)\longrightarrow \mathbb{R}$, define
\[
u_{\min}\left(t\right)=\min_{x\in \overline{\Omega}}u\left(x,t\right) \quad\mbox{and}\quad
u_{\max}\left(t\right)=\max_{x\in \overline{\Omega}}u\left(x,t\right).
\]

Let us then state the results to be
proved in this paper.
We will first prove, in Section \ref{proof:thm1}, the following theorem. 

\begin{theorem}
\label{thm:1Dlogarithmic}
Consider the one dimensional logarithmic diffusion equation
(\ref{eq:boundaryvaluelogarithmic}) with $p=\frac{3}{2}$.
Let $u_0>0$ be such that 
the compatibility condition
\begin{equation*} 
\partial_x \log u_0\left(\pm l\right)= \pm 2\gamma \sqrt{u_0\left(\pm l\right)}
\end{equation*}
 holds. Then, we have the following
(the implicit and explicit constants below may depend on the initial data, on $\gamma$
and on $l$):
\begin{enumerate}
\item[(i)] If $\partial_{xx}\log u_0>0$ and 
 $\gamma>0$ , solutions to the logarithmic diffusion
are global, blow-up in infinite 
time, and in fact $u_{\min}\left(t\right)\sim t$. Furthermore, 
$t^{\frac{3}{2}}\lesssim u_{\max}\left(t\right)\lesssim e^{Mt}$, and the blow-up
profile cannot be flat as $u_{\max}\left(t\right)/u_{\min}\left(t\right)\rightarrow \infty$.

\item[(ii)] If $\partial_{xx}\log u_0<0$ and $\gamma<0$ 
solutions to the logarithmic diffusion
are global, and blow-down in infinite time, that is $u_{\min}\left(t\right)\rightarrow 0$ as $t\rightarrow \infty$,
and in fact $u_{\min}\left(t\right)\lesssim 1/t$. The blow-down can be as fast 
as $e^{-Dt^2}$, for a constant $D>0$.
 
\end{enumerate}
\end{theorem}

Classical solutions to (\ref{eq:boundaryvaluelogarithmic}),
given positive initial data, exist at least for a short time $T>0$, and the compatibility
condition guarantees that it is at least $C^2$ on $\left[-l,l\right]\times\left[0,T\right)$.
The exponent $p=\frac{3}{2}$ in the boundary non linearity can be relaxed a bit. Indeed, using 
Theorem \ref{thm:1Dlogarithmic} and a comparison principle for the logarithmic diffusion equation,  
we will then prove:
\begin{theorem}
\label{thm:1DlogarithmicB}
Let $-\infty < p\leq \frac{3}{2}$, $\gamma>0$, and let $u_0>0$ be such that 
the compatibility condition
$$\partial_x \log u_0\left(\pm l\right)= \pm 2\gamma u_0^{p-1}\left(\pm l\right)$$
holds.
Then, the solution to (\ref{eq:boundaryvaluelogarithmic})
with initial data $u_0$
is global and blows-up in infinite time.
Moreover, if at some point in time $u$ satisfies
$\partial_{xx}\log u > 0$, we have the following blow-up rates:
\begin{itemize}
\item[1.] If $p=2$ then
\[
u_{\max}\left(t\right)\gtrsim e^{\frac{\gamma}{l}t}.
\] 

\item[2.] If $p<2$ then
\[
u_{\max}\left(t\right)\gtrsim t^{\frac{1}{2-p}}.
\]

\end{itemize}
In any case, for $p\leq \frac{3}{2}$ the blow-up can be no faster than
$e^{Mt}$ (that is, any blow-up rate for solutions to
(\ref{eq:boundaryvaluelogarithmic}) with $\gamma>0$ must be $\lesssim e^{Mt}$) where $M>0$ is a constant that only depends on $l$ and $\gamma$.

\end{theorem}

\begin{theorem}
\label{thm:1Dlogarithmicp2}
Let $p\geq \frac{3}{2}$, $\gamma<0$, and let $u_0>0$ be such that 
the compatibility condition
$$\partial_x \log u_0\left(\pm l\right)= \pm 2\gamma u_0^{p-1}\left(\pm l\right)$$
holds.
Then, the solution to
(\ref{eq:boundaryvaluelogarithmic})
with initial data $u_0$
is global and blows-down in infinite time. 
Furthermore, if at some point in time $u$ satisfies
$\partial_{xx}\log u < 0$, we have the following blow-down rates:
\begin{itemize}
\item[1.] If $p=2$ then
\[
u_{\min}\left(t\right)\lesssim e^{\frac{\gamma}{l}t}.
\] 

\item[2.] If $p>2$ then
\[
u_{\min}\left(t\right)\lesssim \frac{1}{t^{\frac{1}{p-2}}}.
\]
\end{itemize}
In any case, for $p\geq \frac{3}{2}$ the blow-down can be no faster than $e^{-Dt^2}$
(that is, any blow-down rate for solutions to
(\ref{eq:boundaryvaluelogarithmic}) with $p\geq \frac{3}{2}$ and $\gamma<0$ must be $\gtrsim e^{-Dt^2}$)
where $D>0$ is a constant that only depends on $l$ and $\gamma$.

\end{theorem}
A proof of Theorems 
\ref{thm:1DlogarithmicB} and \ref{thm:1Dlogarithmicp2}
 will be presented in Section \ref{proof:thm23}.
 
 \bigskip
Theorems \ref{thm:1DlogarithmicB} and \ref{thm:1Dlogarithmicp2} reveal some interesting behavior
of the
logarithmic diffusion equation. For instance, notice that 
when $p=1$ and $\gamma<0$, classical solutions to the logarithmic diffusion equation cannot be global.
Indeed, we can compute
\[
\frac{d}{dt}\int u\left(x,t\right)\,dx = 4\gamma,
\]
which shows that, as $\gamma<0$, blow-down occurs in finite time, and
in fact, as we will show later in Section \ref{blowfinite}, there is finite time blow down for certain initial data
when $\gamma<0$ and $p<1$. Then 
a natural question, which we were unable to resolve,
arises: for $\gamma<0$ are solutions global for $p>1$? or else what is the value of $p_0>1$ such that
for $p>p_0$ (or $p\geq p_0$) solutions are global, and 
for $p\leq p_0$ (resp. $p<p_0$) blow-down occurs in finite time? 
The exponent $p_0$ can be thought of as a critical exponent, and it is called 
a Fujita exponent (related to this, see the interesting work \cite{galaktionov96}).
Something similar happens when $\gamma>0$, though
instead of blow-down we have blow-up: in Section
\ref{blowfinite}, we will also
show that for certain initial data if $p>2$ and $\gamma>0$, there is blow-up in finite time.

\medskip


We shall base our proof of Theorem \ref{thm:1Dlogarithmic}
 on the behaviour of the unnormalised Ricci flow on a cylinder. To be
more precise, we shall deduce Theorem \ref{thm:1Dlogarithmic}
 from the behavior of the following
boundary value problem for a time evolving metric $g\left(t\right)$
on the cylinder $\mathcal{M}=\left[-l,l\right]\times \mathbb{S}^1$:

\begin{equation}
\label{Ricciunnormalised}
\left\{
\begin{array}{l}
\displaystyle\frac{\partial g}{\partial t}=-R_g g \quad \mbox{in}\quad \mathcal{M}\times\left(0,T\right),\\
k_{g}=\gamma \quad \mbox{on}\quad \partial \mathcal{M}\times\left(0,T\right),\\
g\left(\cdot,t\right)=g_0\left(\cdot\right) \quad \mbox{in}\quad \mathcal{M},
\end{array}
\right.
\end{equation}
where $R_g$ and $k_g$ denote the scalar curvature of $\mathcal{M}$ and the
geodesic curvature of 
$\partial \mathcal{M}=\left[\left\{-l\right\}\cup\left\{l\right\}\right]\times
\mathbb{S}^1$ 
with respect to the metric $g$, respectively;
here, $\gamma=k_{g_0}$ is the geodesic curvature of $\partial \mathcal{M}$
with respect to the initial metric. For our results it will suffice 
to consider initial data of the form $g_0=dx^2+f^2\left(x\right)d\theta^2$;
when a cylinder is endowed with a metric of this form, we will say that 
it
is $\mathbb{S}^1$-symmetric or that it has $\mathbb{S}^1$-symmetry.

It might be interesting to observe the 
following. Theorem \ref{thm:1Dlogarithmic},
which is proved using the Ricci flow, is used to prove Theorems
\ref{thm:1DlogarithmicB} and \ref{thm:1Dlogarithmicp2}, and
then we use these theorems as a tool to prove a new longtime existence result
for the Ricci flow on a cylinder, so in some
sense our ideas come full circle. Namely, in Section
\ref{sect:riccilongtime} we shall show:

\begin{theorem}
\label{thm:longtimeexistencericciflow}
Let 
$g_0$ be an arbitrary smooth metric on the cylinder
$\mathcal{M}=\left[-1,1\right]\times \mathbb{S}^1$. Let 
$k_{g_0}\left(\cdot\right)$ be the
geodesic curvature of the boundary with
respect to $g_0$. Let
\[
\varphi: \partial \mathcal{M}\times \left[0,\infty\right)
\longrightarrow \mathbb{R}
\]
be a smooth function such that
$\varphi\left(\cdot,0\right)=k_{g_0}\left(\cdot\right)$, and
which is bounded on any finite
time interval.
Then, the solution
to the unnormalised
Ricci flow 
\begin{equation}
\label{Ricciunnormalised2}
\left\{
\begin{array}{l}
\displaystyle\frac{\partial g}{\partial t}=-R_g g \quad \mbox{in}\quad \mathcal{M}\times\left(0,T\right),\\
k_{g}=\varphi\left(\cdot, t\right) \quad \mbox{on}\quad \partial \mathcal{M}\times\left(0,T\right),\\
g\left(\cdot,t\right)=g_0\left(\cdot\right) \quad \mbox{in}\quad \mathcal{M},
\end{array}
\right.
\end{equation}
 with initial
data $g_0$ exists for all time. 
\end{theorem}

There are very few results on the
 behavior of the Ricci flow with non homogenous boundary conditions
even in the case of surfaces.
The reader might be interested in consulting the works \cite{brendle02, cortissoz16, cortissoz20}. 
We would also like to point out that the study of Ricci flow, and other
geometric flows, on surfaces with boundary might
have some interesting applications on the dynamics of plant growth \cite{almosleh, pulwicki}. 

\subsection{}
We have tried, without
much success, to extend our 
results to higher dimensions. In
any case, in Section \ref{sect:diffusionondisc}, we have included a small incursion into the realm
of the logarithmic diffusion equation in dimension
2. To be more precise we study the behaviour of the logarithmic diffusion
equation
\begin{equation*}
\left\{
    \begin{array}{l}
    \partial_t u =\Delta \log u \quad \mbox{in}\quad D_{a}\times\left(0,T\right),\\
    \dfrac{\partial u}{\partial \eta}=2\gamma u^{p} \quad
    \mbox{on} \quad \partial D_a\times\left(0,T\right),
    \end{array}
\right.
\end{equation*}
where $D_a\subset \mathbb{R}^2$ is the open disc of radius $a$ centered at the origin.
We will show that for $\gamma\leq -\frac{1}{2a}$ and $p\leq 1$, $u$ blows-down in finite time.
This last section is based on the work done in \cite{brendle02, cortissoz20} in the context of the
Ricci flow.

To complete the plan
of the paper, in 
Sections \ref{sect:riccilog} and \ref{sect:riccicyl}, we show explicitly the relation between the Ricci flow and
recall some results on the Ricci flow on a cylinder that might be of interest
to the more geometrically minded reader.

\section{The relation between Ricci flow and 
the logarithmic diffusion equation}
\label{sect:riccilog}

The Ricci flow and the logarithmic diffusion equation are two sides of a same coin. 
For the benefit of the reader, let us spell out the relation between these two equations:
Given two metrics $g$ and $\tilde{g}$ related by a $\tilde{g}=e^{2f}g$
the scalar curvature transforms as
\[
R_{\tilde{g}}=e^{-2f}\left(R_g-2\Delta_g f\right).
\]
Regarding the geodesic curvature of the boundary, if
$N_{g}$ is the outward pointing unit (with respect to $g$) normal vector,
we have the transformation formula
\[
k_{\tilde{g}}=e^{-f}\left(k_g+\frac{\partial f}{\partial N_g}\right),
\]
\emph{where $N_{g}$ is the
outward pointing unit normal with respect to the metric $g$}.
Therefore, when written in terms of the conformal factor, that is, if we write
$\tilde{g}=ug_E$ with $g_E$ such that $R_{g_E}=0$ with totally
geodesic boundary (which can be done if $\mathcal{M}$
is a cylinder), then the Ricci flow equation (\ref{Ricciunnormalised}) becomes
\[
\left\{
\begin{array}{l}
\partial_t u=\Delta_{g_E} \log u\quad \mbox{in} \quad \mathcal{M}\times\left(0,T\right)\\
\dfrac{\partial u}{\partial N_{g_E}}= 
2\gamma u^{\frac{3}{2}}\quad \mbox{on}\quad \partial\mathcal{M}\times\left(0,T\right)\\
u\left(p,0\right)=u_0\left(p\right) \quad \mbox{in}\quad \mathcal{M},
\end{array}
\right.
\]
which is a logarithmic diffusion equation with nonlinear Robin boundary conditions.
In the case that the initial metric is of the form
$g_E=dx^2+f^2\left(x\right)\,d\theta^2$, then the Ricci flow equation
is equivalent to a 1D logarithmic diffusion equation, and viceversa, a 1D logarithmic
equation is equivalent to the Ricci flow equation on a cylinder where the
initial metric and the solution have an $\mathbb{S}^1$ symmetry.

To be more precise, given the 1D 
logarithmic diffusion equation with initial data $u_0$ on $\left[-l,l\right]$, we
can transform it into an
equivalent problem in the Ricci flow setting by considering a flat metric
on the cylinder
$\mathcal{M}=\left[-l,l\right]\times \mathbb{S}^1$, which we shall denote
by $g_E$, and then using the metric $g_0=u_0 g_E$
as the initial data for the Ricci flow equation. For the time evolving metric $g=u g_E$
we have the important relations
\begin{equation}
\label{eq:curvaturefromconformal}
R_{g}=-\frac{\partial_{xx} \log u}{u}=-\frac{u_t}{u}, \quad\quad k_{g}=\gamma,
\end{equation}
and 
\begin{equation}
\label{eq:conformalfromcurvature}
u\left(x,t\right)=u_0\left(x\right)\exp\left(\int_0^{t}R_g\left(x,\tau\right)\,d\tau\right).
\end{equation}

\section{Some results on the Ricci flow on a cylinder}
\label{sect:riccicyl}

Before we begin the task of proving the main results of this paper, let us give a quick survey on 
what is known about the boundary value problem (\ref{Ricciunnormalised}) on a cylinder.
In this section, we will mention the normalised Ricci flow, although it will not be 
directly used
in our arguments. The normalised Ricci flow is obtained from the unnormalised flow by 
rescaling the solution so that the area of the surface remains constant, and then
rescaling the time variable appropriately (see \cite{cortissoz20}).

\subsection{The case $\gamma=0$}

This case was treated by Brendle in \cite{brendle02} using a doubling argument. In this case,
the normalised and unnormalised flow coincide, and there is exponential convergence towards
a flat metric as shown in the case of closed surfaces of Euler characteristic zero in
\cite{Hamilton88}. In terms, of the the logarithmic diffusion equation, this implies
that given initial data $u_0>0$ which satisfies $\partial_x \log u_0\left(\pm l\right)=0$, the solution
to the logarithmic diffusion equation converges exponentially towards a constant.

\subsection{The Case $R\geq 0$, $\gamma\leq 0$}

This case was treated in \cite{cortissoz16}. We summarize some of the results obtained in that paper.
It was shown that on the cylinder 
$\left[-l,l\right]\times \mathbb{S}^1$ for initial data 
\[
g_0=dx^2+f\left(x\right)^2\,d\theta^2
\]
both the normalised and unnormalised flow exists for all time. Then a family of examples
was constructed for which the normalised flow converges towards a flat metric with
totally geodesic boundary, but for which the curvature does not converge exponentially,
again in stark contrast to the case of totally geodesic boundary and
closed surfaces of Euler characteristic zero (\cite{brendle02,Hamilton88}).
That is, if $R_{\max}=\sup_{\mathcal{M}\times \left[0,t\right]}R_g$, then we have that
\[
R_{\max}\left(t\right)\rightarrow 0, \quad \mbox{but}\quad R_{\max}\left(t\right)\geq \frac{2}{t}.
\]
This result might
be of some interest in view of the results 
on the existence of slowly convergent Yamabe flows constructed in \cite{carlotto15}, this
of course in dimensions $n\geq 3$.
It is on some of the results proved in \cite{cortissoz16} that
we base our proof of Theorem \ref{thm:1Dlogarithmic} part (ii).

\section{Some useful formulas and estimates}

In this section we collect some formulas and facts that will be used
to prove the main results of this paper.

We begin with some geometric facts related to a
time evolving metric $g$ that is a solution to the Ricci
flow equation (\ref{eq:unriccifloweq}). First,
we can compute the evolution
equation satisfied by the scalar curvature of the metric $R_g$:
\begin{equation}
\label{eq:unriccifloweq}
\left\{
\begin{array}{l}
\partial_t R_g=\Delta_g R_g + R_g^2 \quad \mbox{in}\quad \mathcal{M}\times\left(0,T\right)\\
\displaystyle\frac{\partial R_g}{\partial N_g}=k_g R_g \quad \mbox{on}\quad \partial \mathcal{M}\times\left(0,T\right),
\end{array}
\right.
\end{equation}
where $N_g$ is the outward pointing unit normal
(for a proof see \cite{cortissoz20}). 
Using 
the parabolic maximum principle and
Hopf's boundary point lemma it is not 
difficult to show that positive and negative curvature are preserved by the flow:
that is, the concavity or convexity of $\log u_0$ is maintained throughout the evolution for solutions to (\ref{eq:boundaryvaluelogarithmic})
with $p=\frac{3}{2}$; this also
holds for (\ref{eq:boundaryvaluelogarithmic}) with
general $p$ as we shall show below.
In what follows $A_g$ denotes the area of the surface $\mathcal{M}$ with respect to the
time evolving metric
$g\left(t\right)$ and $L_g$ denotes the length of the boundary $\partial \mathcal{M}$ also
with respect to the metric $g\left(t\right)$. An important observation
is that if $g=u\left(x,\theta,t\right)g_{E}$, $g_E$ a flat metric on the cylinder with
totally geodesic boundary, then
\begin{equation}
\label{eq:definitionarea}
A_g\left(t\right)= \int_{\mathcal{M}}\,dA_{g} = \int_0^{2\pi}\int_{-l}^{l} u\left(x,\theta,t\right)\, dx\, d\theta.
\end{equation}
In PDE jargon, $A_g$ is called the mass of $u$ and it is denoted by $m$
(see \cite{shimojo18}).
In the case of the unnormalised flow 
the area (mass) satisfies the equation
\begin{equation}
\label{eq:area}
A_g'\left(t\right)=\left(\int_{\mathcal{M}}\,dA_{g}\right)_t= -\int_{\mathcal{M}}R_g\,dA_{g},
\end{equation}
which follows easily from (\ref{eq:definitionarea}) and (\ref{eq:curvaturefromconformal}).

We also have the formula
\begin{equation}
\label{eq:totalcurvature}
\left(\int_{\mathcal{M}}R_g\,dA_{g}\right)_t=\int_{\partial \mathcal{M}}k_g R_g\,ds_{g},
\end{equation}
which follows from the equation satisfied by the curvature (\ref{eq:unriccifloweq}).

Define the boundary average curvature $r_{\partial}\left(t\right)$ as
\[
r_{\partial}\left(t\right)=\frac{\int_{\partial \mathcal{M}}R_g\, ds_g}{L_g},
\]
where $L_g$ is the length of the boundary with respect to the metric $g$.
We can compute
\begin{equation*}
L_g'\left(t\right)=-\frac{1}{2}\int_{\partial\mathcal{M}} R_g\,ds_g
=-\frac{1}{2}r_{\partial}\left(t\right)L_g\left(t\right).
\end{equation*}
Therefore,
\begin{equation}
\label{eq:length}
L_g\left(t\right)=L_g\left(0\right)\exp\left(-\frac{1}{2}\int_0^t r_{\partial}\left(\tau\right)\,d\tau\right).    
\end{equation}
From (\ref{eq:area}) and the Gauss-Bonnet theorem, with $k_g=\gamma$, we obtain
\begin{equation}
\label{eq:arealength}
A_g'\left(t\right)=2\gamma L_g\left(t\right).
\end{equation}

Next, we prove a geometric estimate
relating the area of a cylinder and the
length of its boundary (Lemma 2.1 in \cite{cortissoz16}).

\begin{lemma}
\label{lemma:areatolength}
Consider a metric (not necessarily time evolving) on the cylinder $\left[-l,l\right]\times \mathbb{S}^1$ of the form
\[
g=dx^2+f^2\left(x,\theta\right)\,d\theta^2
\]
Assume that the curvature satisfies $R_g\geq 0$ and the geodesic curvature
of the boundary can be bounded $\left|k_g\right|\leq \alpha$. Then
the following estimate holds
\[
A_g\leq \frac{2L_g}{\alpha}\sinh{\alpha l}.
\]
\end{lemma}
\begin{proof}
For a metric of the form $dx^2 + f^2\left(x,\theta\right)d\theta^2$, define
$\varphi=\dfrac{f'}{f}$, where the prime denotes
differentiation with respect to $x$. Then we have that
\[
\varphi'=-\frac{R_g}{2}-\varphi^2.
\]
Therefore, 
\[
\varphi\left(x,\theta\right)\leq \alpha,
\]
and after a second integration
\[
f\left(x,\theta\right)\leq \frac{1}{\alpha}f\left(-l,\theta\right)e^{\alpha r},
\]
and hence,
\[
A_g=\int_0^{2\pi}\int_{-l}^{l}f\left(x,\theta\right)\,dr\,d\theta\leq \frac{2L_g}{\alpha}\sinh\left(\alpha l\right),
\]
and the lemma is proved.
\end{proof}

We finish this section by showing, as promised above, that 
the convexity (concavity) of $\log u_0$ is preserved by the solution
to (\ref{eq:boundaryvaluelogarithmic}) with initial data $u_0$.
\begin{lemma}
\label{lemma:preservingsign}
Let $p$ be arbitrary, and
let $u_0>0$ be such that the compatibility condition 
$$\partial_x \log u_0\left(\pm l\right)= \pm 2\gamma u_0^{p-1}\left(\pm l\right)$$
holds.
Let $u>0$ be the solution to (\ref{eq:boundaryvaluelogarithmic})
with initial data $u_0$. Then if $\partial_{xx}\log u_0 >0$
($\partial_{xx}\log u_0<0$), $\partial_{xx}\log u >0$
(resp. $\partial_{xx}\log u< 0$)
as long as the solution exists.
 \end{lemma}
 
\begin{proof}
Let $R=-u_t/u$. Then, $R$ satisfies the following parabolic evolution equation
\begin{equation*}
\left\{
\begin{array}{l}
\partial_t R= \partial_{xx} R + R^2 \quad \mbox{in}\quad \left(-l,l\right)\times\left(0,T\right)\\
\partial_x R\left(\pm l, t\right) =\pm 2\gamma\left(p-1\right)u^{p-1}R. 
\end{array}
\right.
\end{equation*}
Let us show that $\partial_{xx}\log u_0 <0$ then it remains so. Indeed, 
the initial condition implies that
at $t=0$, $R>0$, and the compatibility condition shows that
$R$ is continuous in $\left[-l,l\right]\times\left[0,T\right)$. 
But then if at a future time $t>0$ we have that $R=0$, it
would be a minimum, and by Hopf's boundary point lemma there we must have that 
$\dfrac{\partial R}{\partial \eta}<0$, where $\eta$ represents the outward unit
exterior normal, which at $l$ is $\partial_x$ and at $-l$ is $-\partial_x$. But
if $R=0$ then $\dfrac{\partial R}{\partial \eta}=0$, which is a contradiction.
Hence the point $p$ where $R\left(p,t\right)=0$ is located in $\left(-l,l\right)$;
if such is the case, the parabolic maximum principle implies that a nonnegative maximum occurs 
in the interior only if $R$ is constant in space and time. This shows our claim.

\medskip
The proof that $\partial_{xx}\log u_0>0$ is preserved follows similar arguments, and
is left to the interested reader.
\end{proof}

\subsection{A comparison principle for the logarithmic diffusion equation}

The following is a basic, classical, and quite useful comparison principle for solutions to the logarithmic diffusion
equation.
\begin{theorem}
Assume that $v$ and $u$ are strictly positive solutions to the logarithmic diffusion equation 
in $C^2\left(\overline{\Omega} \times \left(0,T\right)\right)\cap 
C\left(\overline{\Omega} \times \left[0,T\right)\right)$, $\Omega$ an open bounded subset of
$\mathbb{R}^n$, with
initial data $v_0>0$ and $u_0>0$ respectively. Assume that $v_0>u_0$,
and that at the boundary
we have inequalities $\partial v/\partial \eta\geq f\left(v\right)$ 
and $\partial u/\partial \eta \leq f\left(u\right)$.
Then $v> u$ for all $t\in \left(0,T\right)$.
\end{theorem}
\begin{proof}
First notice that 
\[
\Delta \log v-\Delta \log u =\zeta\left(u,v\right)\Delta\left(v-u\right)
+2\nabla\zeta\left(u,v\right)\cdot \nabla \left(v-u\right)+
\left(v-u\right)\Delta\zeta\left(u,v\right),
\]
where
\[
\zeta\left(u,v\right)=\int_0^1 \frac{1}{\left(1-t\right)v+tu}\,dt>0.
\]
Hence, $z=v-u$ satisfies the parabolic evolution problem
\[
\left\{
\begin{array}{l}
\partial_t z = a\Delta z + \nabla \mathbf{b}\cdot \nabla z+ cz
\quad \mbox{in}\quad \Omega\times \left(0,T\right)\\
z_0=v_0-u_0 >0 \quad \mbox{in} \quad \Omega.
\end{array}
\right.
\]
Assume that at some $t>0$ we have $z=0$. Then, the parabolic maximum principle
implies that this happens at a boundary point
$p\in \partial \Omega \times \left\{t\right\}$
. Hence, Hopf's boundary point lemma
implies that at this point $\partial z/\partial \eta \left(p\right)<0$. But this contradicts
our hypothesis. Hence, at all times we must have $z>0$, and the theorem is proved.
\end{proof}

\medskip
\noindent
{\bf Remark.} $\Omega$ can be substituted in the statement by a compact manifold with boundary,
and in this case $\eta$ represents the exterior unit normal to the boundary with respect to the metric.

\section{Proof of Theorem \ref{thm:1Dlogarithmic}}
\label{proof:thm1}

\subsection{Proof of Part (i)}

For $\gamma>0$ and initial data as in the hypothesis of part
(i) of Theorem \ref{thm:1Dlogarithmic},
we shall first prove global existence. 
A particular aspect is of this proof is that it uses
a compactness theorem for the Ricci flow on manifolds with 
boundary due to Gianniotis \cite{Gianniotis16}. Then with the aid of
some geometric estimates we shall give (non-optimal) blow-up rates 
for these global solutions.
\subsubsection{Global Existence}
\label{globalexistence}
In this section we show that in the unnormalised Ricci flow, if we start with
a metric on $\left[-l,l\right]\times\mathbb{S}^1$ of the form
\[
g_0=dx^2+f^2\left(x\right)d\theta^2,
\]
and whose curvature is negative, then the curvature remains uniformly bounded throughout 
the flow.
This in turn implies that the unnormalised flow exists for all time, and so does 
the solution to the corresponding 1D logarithmic diffusion equation.
The proof, which is a simple application of Gianniotis' Compactness 
Theorem (\cite{Gianniotis16}, see also \cite{cortissoz20}), goes as follows. 

Assume that at the boundary there is a sequence of times such that
$\left|R\right|$ blows up. 
Blow-up occurs at maximal rate at $\partial \mathcal{M}$ because 
if $R_{\min}\left(t\right)$ occurs at an interior point, the
maximum principle implies that at that instant $R_{\min}\left(t\right)$
is nondecreasing.
The needed
hypothesis for the application of Gianniotis' compactness theorem are verified just as it is done
in \cite{cortissoz20}; in this case, we do not need a bound
on the injectivity radius because we shall take the
blow up limit at a boundary point, and
in our case the focal radius and the boundary injectivity radius estimates are easy to obtain due to the simmetries,
and the fact that under the Ricci flow if $R_g<0$ distances are increasing.
Indeed, in the case we are treating, 
any geodesic issuing orthogonally from a boundary component
is minimizing and it hits the other boundary component orthogonally.
Hence, the focal radius is equal to the length of any
of these geodesics, and 
the boundary injectivity radius
is given by half the length of the shortest geodesic
joining the two boundary components of the cylinder; but since $R_g<0$, the lengths of these gedesics is increasing,
 which
gives us the desired estimate on the boundary injectivity radius as claimed.

Next, if we let $M\left(t\right)=\sup_{\left[0,t\right]}\sup_{x\in \mathcal{M}} \left|R_g\left(x,t\right)\right|$,
then it is clear that
\[
\sup tM\left(t\right)=\infty.
\]
Thus, if we form a blow up limit, we obtain an eternal solution which is a half plane with totally geodesic boundary.
By doubling this solution we then obtain an eternal solution to the Ricci flow
on $\mathbb{R}^2$ (see Appendix B.1 in \cite{cortissoz20}),
and this would be an eternal solution to the Ricci flow with nonpositive curvature;
but this cannot be,
as any eternal solution to the Ricci flow has nonnegative curvature 
(as explained in (vi), page 3, in \cite{daskalopoulos06}).
From this we deduce that
the curvature must remain uniformly bounded
along the flow, even if it is global.
As an aside, this is the behaviour of the curvature in the case of the unnormalised flow 
in the case of a compact hyperbolic surface.

This shows that, given
an initial metric with $\mathbb{S}^1$ symmetry and 
of negative curvature, 
the unnormalised Ricci flow (\ref{Ricciunnormalised}) exists for all time, 
and hence the solution
to the logarithmic diffusion equation (\ref{eq:boundaryvaluelogarithmic})
with initial data $u_0$ such that $\partial_{xx}\log u_0<0$ and
which satisfies the compatibility condition. 

\subsubsection{Blow-up}

Because 

\[
A_g''\left(t\right) = -\int_{\partial \mathcal{M}} \gamma R_g\ ds_{g} \geq 0,
\]
by integrating twice we obtain

\begin{equation*}\label{left_area}
A_g\left(t\right) \geq A_g'(0)t + A_g(0),
\end{equation*}
which proves that the area grows at least linearly. 
However, as stated in the theorem, we can do better.
First notice that as $$\displaystyle \frac{\partial R_g}{\partial N_g}=\gamma R_g <0,$$
the maximum of the curvature at time $t>0$ must lie in the interior of $\mathcal{M}$.
Hence, $R_{\max}\left(t\right)$
satisfies a differential inequality
\[
\frac{d R_{\max}\left(t\right)}{dt}\leq R_{\max}^2.
\]
Therefore, writing $B=R_{\max}\left(0\right)<0$,
\begin{equation}
\label{ineq:maxcurvature}
R_{\max}\left(t\right)\leq \frac{B}{1-Bt},
\end{equation}
and from this we deduce
\[
-\frac{1}{2}r_{\partial}\geq -\frac{1}{2}\frac{B}{1-Bt}.
\]
But as from (\ref{eq:arealength}) we have
\[
A_g'\left(t\right)\gtrsim \exp\left(-\frac{1}{2}\int_0^t r_{\partial}\left(\tau\right)\,d\tau\right)
\gtrsim \int_0^{t}\sqrt{1-B\tau}\,d\tau,
\]
by integration we obtain
\begin{equation}
\label{ineq:areagrowthbelow}
A_g\left(t\right)\gtrsim t^{\frac{3}{2}}.
\end{equation}

A lower estimate $u_{\max}$ follows from the fact that
\[
t^{\frac{3}{2}}\lesssim A_g\left(t\right)\lesssim 
\int_{-l}^l u\left(x,t\right)\,dx \leq 2l \cdot u_{\max}\left(t\right).
\]
However, as we shall see below, this blow-up rate is not optimal.

To obtain an upper estimate
for the blow-up rate, notice that the argument in Section
\ref{globalexistence} actually shows that 
$R$ is uniformly bounded on the
interval of existence of 
the solution, that is, there is an $M$ such that $-M\leq R <0$ 
for all time. But since
$R=-u_t/u$, by integration we obtain that $u\lesssim e^{Mt}$.

On the other hand, a lower bound for $u_{\min}$ follows from 
\[
u_{\min}\left(t\right)\geq u_{0,\min}e^{-\int_0^t R_{\max}\left(\tau\right)\,d\tau}\gtrsim t,
\]
and, by the Topping-Yin interior estimate
\cite{topping17}, we must actually have that $$u_{\min}\left(t\right)\sim t.$$

Estimate (\ref{ineq:areagrowthbelow}) is rather interesting
in view of the following two facts: in the case of hyperbolic compact surfaces, under the unnormalised flow,
the area grows linearly, and in the case of surfaces of Euler characteristic
$\chi\left(\mathcal{M}\right)=0$ the area remains constant.

\hfill $\Box$

\subsection{Proof of Part (ii)}
This part of Theorem \ref{thm:1Dlogarithmic} follows from the results in \cite{cortissoz16}.
Indeed, if we have initial data satisfying $\partial_{xx} \log u_0 <0$, this implies that
the corresponding metric $u_0 g_E$ has positive curvature, and thus we are in the case
studied in \cite{cortissoz16}, and
hence, global existence follows from Theorem 1.1 in \cite{cortissoz16}. 
Blow-down is a consequence of global existence plus the 
existence of a Lyapunov functional since, from the maximum principle
and Hopf's boundary point lemma, $u$ is bounded above.
However, with a little more work we can give a blow-down rate. Indeed, by Theorem 4.1
in \cite{cortissoz16}, we have 
that 
\[
\int_{\mathcal{M}}R_g\, dA_g \lesssim \frac{1}{t}.
\]
Therefore, by the Gauss-Bonnet theorem
\[
-\int_{\partial \mathcal{M}}\gamma \, ds_g \lesssim \frac{1}{t},
\]
and, since $\gamma<0$, $L\left(t\right)\lesssim 1/t$. But by Lemma \ref{lemma:areatolength},
and the fact that distances are diminishing as the curvature is positive, we 
necessarily have that
\[
A_g\left(t\right)\lesssim L_g\left(t\right)\lesssim\frac{1}{t}.
\]
A blow-down rate follows immediately as
\[
u_{\min}\left(t\right)\lesssim \int_{-l}^{l}u\left(x,t\right)\,dx \lesssim 
A_g\left(t\right).
\]

\subsubsection{Examples of fast blow-down rate} 
\label{ssection:uniformblowdown}
Now, we construct initial conditions for (\ref{eq:boundaryvaluelogarithmic}),
with $p=\frac{3}{2}$, so that the solution has a fast blow-down rate as
stated in Theorem \ref{thm:1Dlogarithmic} part (ii). In order to do this, let us recall that
in Section 4.1 of \cite{cortissoz16}, examples are suggested to show that
the curvature converges uniformly to 0 in the case of the
normalised flow.
We shall show next how to construct these examples.
Consider the cylinder $\left[-l,l\right]\times\mathbb{S}^1$ with the following metric
\[
g_0=dx^2+\left(\cos x-0.25 x^2\right)^2\,d\theta^2.
\]
For the choice $l\sim 0.74013$, the solution to the Ricci flow equation
with this initial data is at least $C^3$ in $\overline{\mathcal{M}}\times\left[0,T\right)$. 
Indeed, for this value of $l$ we obtain that the initial metric
satisfies the compatibility condition
\[
\frac{\partial R_{g_0}}{\partial N_{g_0}}=k_{g_0} R_{g_0},
\]
which guarantees that the solution to the equation satisfied by 
the scalar curvature is at least $C^1$ in $\overline{\mathcal{M}}\times\left[0,T\right)$,
which in turn translates into the regularity claimed for the solution to the Ricci flow
with initial data $g_0$.
It can be easily shown that 
the curvature for $g_0$ is decreasing from the middle parallel, that
is 
from $\left\{0\right\}\times \mathbb{S}^1$, towards the boundary components.
Then, we can differentiate through the equation and use
the maximum principle to show that the curvature
decreases from the middle parallel towards the boundary
throughout the flow.
For this 
type of example, it is shown in \cite{cortissoz16} that the solution to the unnormalised Ricci
flow with initial data $g_0$ satisfies
\[
\int_0^{\infty}R_{\max}\left(\tau\right)-R_{\min}\left(\tau\right)\,d\tau<\infty.
\]
Let $u_0$ be such that $g_0=u_0 g_E$, where $g_E$ is a flat metric on the
cylinder with totally geodesic boundary. Notice that $\partial_{xx}\log u_0 >0$ as $R_{g_0}<0$. 
It follows then that there are solutions to 
(\ref{eq:boundaryvaluelogarithmic}) with $\gamma<0$ and initial data $u_0$ which
satisfies that 
\[
0<\frac{u_{\max}\left(t\right)}{u_{\min}\left(t\right)}\leq C,
\]
for a constant $C>0$, and thus for this family of examples the blow-down is uniform,
that is, both $u_{\max}$ and $u_{\min}$ converge to zero
as $t\rightarrow \infty$ at the same rate.
But then, for these examples one can show that the blow-down rate is quite fast.
Indeed, by Theorem 1.3 in \cite{cortissoz16}, we have that
\[
R_{\max}\left(t\right)\gtrsim t,
\]
and hence, from (\ref{eq:conformalfromcurvature}),
\[
u_{\min}\left(t\right)\lesssim e^{-Dt^2},
\]
for a constant $D>0$.

We must observe that by properly rescaling $u_0$ we can find initial conditions
$u_{0,l,\gamma}$
on any interval $\left[-l,l\right]$ such that for a given $\gamma<0$,
$\partial_x \log u_{0,l,\gamma}\left(\pm l\right)=\pm 2\gamma \sqrt{u_{0,l,\gamma}\left(\pm l\right)}$.
So we have shown,
\begin{prop}
\label{thm:examplesuniformblowdown}
For $\gamma>0$ and $l>0$
there are initial data defined on $\left[-l,l\right]$ such that the solution to 
(\ref{eq:boundaryvaluelogarithmic}) on $\left[-l,l\right]\times \left(0,T\right)$
are global (i.e., $T=\infty$),
 and blow down uniformly, in the sense 
 described above, at rate at least $e^{-Dt^2}$, for 
a constant $D$ that only depends on the initial data and on $\gamma$ and $l$.
\end{prop}

\section{Proof of Theorems \ref{thm:1DlogarithmicB} and \ref{thm:1Dlogarithmicp2}}
\label{proof:thm23}
\subsection{A proof of Theorem \ref{thm:1DlogarithmicB}}
Let $\delta\geq 0$ and let u be the solution to (\ref{eq:boundaryvaluelogarithmic})
with $p=\frac{3}{2}-\delta$, and with
initial data $u_0>0$.
Choose {\bf any} $v_0$
which satisfies
the assumptions of Theorem \ref{thm:1Dlogarithmic} (i) 
(including the compatibility condition),
and let $v$ be the solution
to (\ref{eq:boundaryvaluelogarithmic}) with $p=\frac{3}{2}$ and
initial data $v_0$. 
Then, let $T>0$ be such that for all $t\geq T$
$$v\left(\cdot, t\right)\geq \max\left\{1, \max_{\left[-l,l\right]}u_0\right\},$$
which does exist since $v\gtrsim t$, and define
$w_0=v\left(\cdot, T\right)$. Let $w$ be the solution to the
logarithmic diffusion equation with $p=\frac{3}{2}$ and initial data $w_0$,
that is $w\left(\cdot,t\right)=v\left(\cdot,t+T\right)$. We
want to apply the comparison principle to 
$z=w-u$.
Denoting by $\partial/\partial\eta$ the operators $\partial_x$ at $x=l$ and $-\partial_x$ at $x=-l$ (i.e., the
partial derivative with respect to the the exterior
normal), since $w\geq 1$, we have that 
\begin{eqnarray*}
\frac{\partial w}{\partial \eta}= 2\gamma w^{\frac{3}{2}}
= 2\gamma w^{\delta}w^{\frac{3}{2}-\delta}
\geq 2\gamma w^{\frac{3}{2}-\delta},
\end{eqnarray*}
and then the comparison principle shows that
$z > 0$, which in turn shows that on any finite time interval $u$ remains bounded above; as
the maximum principle also implies that $u>0$ as long as it does exist, we have that 
the solution is global. Incidentally, this also shows that the blow-up rate, if 
blow-up occurs as we show below, cannot be faster
than $e^{Mt}$ with $M>0$ a constant that only depends on $l$ and $\gamma$.

To show that blow-up must occur, assume that $u$ remains uniformly bounded above. Clearly
$u$ is uniformly bounded away from 0, by the maximum principle and Hopf's boundary point lemma. Then,
the existence of a Lyapunov functional (see \cite{lappicy19}),
shows that along any sequence of times $t_k\nearrow \infty$,
the sequence $u_k:=u\left(\cdot, t_k\right)$ must converge to an equilibrium, say $u_{\infty}$. 
That is, $u_{\infty}$ must satisfy
\[
\partial_{xx}\log u_{\infty} =0,\quad \partial_x u_{\infty}\left(\pm l\right)=\mp 
\gamma u_{\infty}^{p}\left(\pm l\right),
\]
but then an integration by parts shows that this is impossible for $u_{\infty}>0$.

\subsubsection{Blow-up rates}
Let $\gamma>0$ and $p<2$.
Assume that at some time $t_0\geq 0$ $\partial_{xx}\log u>0$;
then for all $t>t_0$ this will continue to hold by Lemma \ref{lemma:preservingsign}, so we shall assume that
$t_0=0$.  
To estimate the blow-up rate, let us define 
\[
r_n\left(t\right)=\int_{-l}^{l}u\left(x,t\right)\,dx.
\]
We estimate its derivative as follows
\begin{eqnarray*}
r_n'\left(t\right)&=&n\int_{-l}^{l}u^{n-1}u_t\,dx = 
n\int_{-l}^{l}u^{n-1}\partial_{xx}\log u\,dx\\
&=&2\gamma n\left[u^{n+p-2}\left(l\right)+u^{n+p-2}\left(-l\right)\right]\\
&\geq &2\gamma n u_{\max}^{n+p-2}\left(t\right)=
2\gamma n \left(u_{\max}^{n}\left(t\right)\right)^{1+\frac{p-2}{n}}\\
&\geq& \frac{2\gamma n}{\left(2l\right)^{1+\frac{p-2}{n}}}
r_n^{1+\frac{p-2}{n}}\left(t\right),
\end{eqnarray*}
where we have used the fact that $\partial_{xx}\log u >0$
implies that the maximum of $u$ at time $t$ occurs at either $x=l$ or $x=-l$.
Solving this differential inequality we get
\[
r^{\frac{2-p}{n}}_n\left(t\right)\geq
r^{\frac{2-p}{n}}_n\left(0\right)+
\frac{2\left(2-p\right)\gamma}{\left(2l\right)^{1+\frac{p-2}{n}}}t.
\]
Hence, we have
\[
u_{\max}^{2-p}\left(t\right)\geq 
r^{\frac{2-p}{n}}_n\left(0\right)+
\frac{2\left(2-p\right)\gamma}{\left(2l\right)^{1+\frac{p-2}{n}}}t.
\]
Letting $n\rightarrow \infty$ yields,
\[
u_{\max}\left(t\right)\geq \left(u^{2-p}_{\max}\left(0\right)
+\frac{\gamma\left(2-p\right)}{l}t\right)^{\frac{1}{2-p}}\sim t^{\frac{1}{2-p}}.
\]

For $p=2$, the equation we obtain is
\[
r_n'\left(t\right)\geq \frac{\gamma n}{l}r_n\left(t\right),
\]
from which $u_{\max}\left(t\right)\gtrsim e^{\frac{\gamma}{l}t}$ follows.

Observe that we only know that there is global existence of solutions if $p\leq \frac{3}{2}$,
and that our estimate on maximum blow-up rate is for $p\leq \frac{3}{2}$ so it does not
cover exponents $\frac{3}{2}<p\leq 2$.

\subsection{A proof of Theorem \ref{thm:1Dlogarithmicp2}}
We only give an sketch as the proof follows the same ideas as the proof of Theorem
\ref{thm:1DlogarithmicB} previously presented. Indeed, for $\delta\geq 0$
let $p=\frac{3}{2}+\delta$, and $u_0>0$ be such that $\partial_{xx}\log u_0 <0$.
Let $u$ be the solution to (\ref{eq:boundaryvaluelogarithmic}), with 
$p=\frac{3}{2}+\delta$ and initial data $u_0$.
For a given $\gamma$, take {\bf any} 
$v_0>0$ constructed as in Proposition \ref{thm:examplesuniformblowdown} in Section \ref{ssection:uniformblowdown},
and satisfying the hypotheses of Theorem
\ref{thm:1Dlogarithmic}, and 
let $v$ be the solution
to (\ref{eq:boundaryvaluelogarithmic}) with 
$p=\frac{3}{2}$ and initial data $v_0$. Hence, as the solution
associated to $v_0$ blows down uniformly, there is a $T>0$ such that
for all $t\geq T$
\[
v\left(\cdot, t\right)\leq \min\left\{1, \min_{\left[-l,l\right]} u_0\left(x\right)\right\}
\]
Just as before, we let $w$ be the solution to (\ref{eq:boundaryvaluelogarithmic}) with initial data 
$w_0=v\left(\cdot, T\right)$, an define $z=u-w$. Again, $z$ satisfies
a semilinear parabolic equation, and at the boundary we have
\begin{eqnarray*}
\frac{\partial w}{\partial \eta}=2\gamma w^{\frac{3}{2}}
= 2\gamma w^{-\delta}w^{\frac{3}{2}+\delta}
\leq 2\gamma w^{\frac{3}{2}+\delta}.
\end{eqnarray*}
Therefore, as at $t=0$ we have that $z>0$ the comparison principle implies that 
$z>0$ for all time, that is $u>v>0$. Of course, this previous argument
also shows that the blow-down rate cannot be faster than $e^{-Dt^2}$, for
$D$ a constant that only depends on $l$ and $\gamma$. As the maximum principle and
Hopf's boundary point lemma also imply that $u$ is bounded above,
global existence follows. Finally, the existence of a Lyapunov functional shows that 
blow-down must occur.


\subsubsection{Blow-down rates}
We let $\gamma<0$ and assume that the initial data
satisfies $\partial_{xx}\log u_0 < 0$ and the compatibility
condition.
We start with $p>2$.
Take any $n>p-2$ and define
\[
q_n\left(t\right)=\int_{-l}^{l}\frac{1}{u^n}\,dx.
\]
We compute its derivative with respect to time
\begin{eqnarray*}
q_n'\left(t\right)&=&-n\int_{-l}^{l}\frac{u_t}{u^{n+1}}\,dx=
-n\int_{-l}^{l}\frac{\partial_{xx}\log u}{u^{n+1}}\,dx\\
&=&-2\gamma n\left[u^{p-2-n}\left(l\right)+u^{p-2-n}\left(-l\right)\right]\\
&\geq& -2\gamma n u_{\min}^{p-2-n}\geq -2n\gamma 
\left(\frac{1}{2l}\int_{-l}^l u^{-n}\right)^{\frac{n+2-p}{n}}\\
&=&\frac{-2n\gamma}{\left(2l\right)^{\frac{n+2-p}{n}}}q_n^{1-\frac{p-2}{n}}.
\end{eqnarray*}
Let
\[
\beta=\frac{-2n\gamma}{\left(2l\right)^{\frac{n+2-p}{n}}}>0.
\]
Solving the differential inequality, we get
\[
\frac{n}{p-2}\left[q_n^{\frac{p-2}{n}}\left(t\right)-
q_n^{\frac{p-2}{n}}\left(0\right)\right]\geq \beta t, 
\]
and hence
\[
\frac{1}{u_{\min}^{p-2}\left(t\right)}\geq q_n^{\frac{p-2}{n}}\left(0\right)+
\frac{\beta\left(p-2\right)}{n} t,
\]
which is equivalent to
\[
u_{\min}\left(t\right)\leq \frac{1}{\left[q_n^{\frac{p-2}{n}}\left(0\right)+
\frac{\beta\left(p-2\right)}{n} t\right]^{\frac{1}{p-2}}}.
\]
Finally, by letting $n\rightarrow\infty$ we have
\[
u_{\min}\left(t\right)\leq \frac{1}{\left[u_{\min}^{2-p}\left(0\right)+
\frac{-\gamma\left(p-2\right)}{l} t\right]^{\frac{1}{p-2}}}\sim 
\frac{1}{t^{\frac{1}{p-2}}}.
\]

Next we treat the case $p=2$. Notice that in this case, 
the differential inequality is
\[
q_n'\left(t\right)\geq \frac{-\gamma n}{l}q_n,
\]
which gives,
\[
q_n\left(t\right)\geq q_n\left(0\right)e^{-\frac{\gamma n}{l}t},
\]
that is,
\[
\frac{1}{u_{\min}\left(t\right)}\geq q_n\left(0\right)^{\frac{1}{n}}
e^{-\frac{\gamma}{l}t}.
\]
Again, taking $n\rightarrow \infty$ yields,
\[
u_{\min}\left(t\right)\leq u_{\min}\left(0\right)e^{\frac{\gamma}{l}t}.
\]

\section{Blow-down and Blow-up in finite time}
\label{blowfinite}

The results in this section complement those of Theorems \ref{thm:1DlogarithmicB}
and \ref{thm:1Dlogarithmicp2} by showing cases when finite time blow-down and blow-up
occur for the logarithmic diffusion equation.

\subsection{Finite time blow-down}
\label{subsect:finiteblowdown}
Let $\gamma<0$, and let $u_0$ such that 
$\partial_{xx}\log u_0<0$ and 
$\partial_x \log u_0\left(\pm l\right) = \pm 2\gamma u_0^{p-1}\left(\pm l\right)$. If $p<1$ then
the solution to (\ref{eq:boundaryvaluelogarithmic}) blows down 
in finite time.
Indeed, recall that the mass of $u$ is given by 
\[
m\left(t\right)\colon =\int_{-l}^{l}u\left(x,t\right)\,dx.
\]
If we write $p=1-\epsilon$, $\epsilon>0$, using that $u_{\min}\left(t\right)$
must be reached at either $x=-l$ or $x=l$ (or at both points), because $\partial_{xx}\log u<0$ holds for the
solution with initial data $u_0$ (i.e., positive curvature is preserved, Lemma 
\ref{lemma:preservingsign}), we can estimate
\begin{eqnarray*}
m'\left(t\right)&=&2\gamma\left(u^{-\epsilon}\left(l\right)+u^{-\epsilon}\left(-l\right)\right)\\
&\leq& 2\gamma u_{\min}^{-\epsilon}\left(t\right)\leq 
\frac{2\gamma}{\left(2l\right)^{-\epsilon}} m^{-\epsilon}\left(t\right).
\end{eqnarray*}

By integration we obtain
\[
\frac{1}{1+\epsilon}\left(m^{1+\epsilon}\left(t\right)-m^{1+\epsilon}\left(0\right)\right)
\leq \frac{2\gamma}{\left(2l\right)^{-\epsilon}}t,
\]
and thus
\[
m^{1+\epsilon}\left(t\right)\leq m^{1+\epsilon}\left(0\right)+
\alpha t,
\quad \alpha=\frac{2\gamma\left(1+\epsilon\right)}{\left(2l\right)^{-\epsilon}}<0,
\]
which shows that there must be blow-down in finite time.

\subsection{Finite time blow-up}
If $\gamma>0$, $\partial_{xx}\log u_0>0$, and $p>2$ there is blow up 
in finite time. Indeed, as $\partial_{xx}\log u>0$, $u_{\max}\left(t\right)$
is reached at either $x=-l$ and/or $x=l$. 
We can then calculate
\begin{eqnarray*}
m'\left(t\right)&=&2\gamma\left(u^{p-1}\left(l\right)+u^{p-1}\left(-l\right)\right)\\
&\geq& 2\gamma u_{\max}^{p-1}\left(t\right)\geq \frac{2\gamma}{\left(2l\right)^{p-1}}
m^{p-1}\left(t\right).
\end{eqnarray*}
By integration we obtain
\[
m^{p-2}\left(t\right)\geq \frac{1}{m^{2-p}\left(0\right)+\beta t},\quad
\beta=\frac{2\left(2-p\right)\gamma}{\left(2l\right)^{p-1}}<0,
\]
which shows our claim. The reader must compare the work done in this section to the
work done in \cite{fila91}.

\section{Proof of Theorem \ref{thm:longtimeexistencericciflow}}
\label{sect:riccilongtime}

Next, we shall make use of the ideas used in proving Theorems 
\ref{thm:1DlogarithmicB} and \ref{thm:1Dlogarithmicp2} to prove 
the longtime existence result for the Ricci flow on a cylinder
stated in Theorem \ref{thm:longtimeexistencericciflow}.

To begin with our arguments, let $g_0$ be an arbitrary metric on
the cylinder $\mathcal{M}=\left[-l,l\right]\times \mathbb{S}^1$, and let
$u_0$ be the conformal factor such that $g_0=u_0 g_E$, where
$g_E$ is a flat metric on the cylinder with totally geodesic boundary.
Then, if $g$ is the solution to the Ricci flow equation (\ref{Ricciunnormalised2})
on the cylinder, with initial data $g_0$, we can write
$g=ug_E$, and $u$ is the solution to the logarithmic 
diffusion equation
\[
\left\{
\begin{array}{l}
\partial_t u = \Delta \log u, \quad \mbox{on}\quad \mathcal{M}\times\left[0,T\right)\\
\dfrac{\partial u}{\partial N_{g_E}}=2\varphi\left(\pm l,\theta,t\right)u^{\frac{3}{2}},
\quad \theta\in \mathbb{S}^1\\
u\left(x,\theta,0\right)=u_{0}\left(x,\theta\right), \quad
\left(x,\theta\right)\in \mathcal{M},
\end{array}
\right.
\]
where $\Delta=\partial_{xx}+\partial_{\theta\theta}$, and $\left[0,T\right)$
is the maximal time of existence for the solution to (\ref{Ricciunnormalised2}).
If we assume that $k_{g_0}=\varphi\left(\cdot, 0\right)$, then the
solution to (\ref{Ricciunnormalised2}) is at least $C^2$ on
$\overline{\mathcal{M}}\times\left[0,T\right)$, and $u$ has the same regularity.

Let
$\gamma=
\max_{\mathcal{M}\times\left[0,T\right]}\varphi$. If $\gamma\leq 0$, the
maximum principle and Hopf's boundary point lemma imply that $u$ is bounded
above. Hence, assume $\gamma>0$, and
let $v$ be solution to (\ref{eq:boundaryvaluelogarithmic}) with $p=\frac{3}{2}$,
$\gamma$ as defined above, and
initial data $v_0\left(x\right)$ such that
$$\min_{x\in\left[-l,l\right]} v_0\left(x\right)>
\max_{\left(x,\theta\right)\in\left[-l,l\right]\times\mathbb{S}^1} u_0\left(x,\theta\right),$$
and which satisfies the compatibility condition. This guarantees that $v$ is also $C^2$
on $\left[-l,l\right]\times\left[0,T\right)$.
 It is not difficult to see that $v\left(x,t\right)$ can be transformed
into a solution $w\left(x,\theta,t\right)$ to the logarithmic diffusion
equation in $\mathcal{M}=\left[-1,1\right]\times \mathbb{S}^1$ by defining
\[
w\left(x,\theta,t\right)=v\left(x,t\right).
\]
In this case, at the boundary we have
\[
\frac{\partial w}{\partial N_{g_E}}=2\gamma w^{\frac{3}{2}}.
\]
The comparison principle then shows that
$u\leq w$ on $\mathcal{M}\times\left[0,T\right)$, and by Theorem
\ref{thm:1DlogarithmicB}, we can conclude that $u$ remains bounded above
on $\left[0,T\right)$.

In a similar fashion we can show that $u$ remains bounded away from 0.
This time, we let $\gamma=\min_{\mathcal{M}\times\left[0,T\right]}\varphi$;
if $\gamma>0$ there is nothing to prove, for then the parabolic
maximum principle and 
Hopf's boundary point lemma show
that the solution $u$ remains bounded below away from 0; so assume that $\gamma<0$.
Pick any solution to (\ref{eq:boundaryvaluelogarithmic}) with $p=\frac{3}{2}$ 
and $\gamma$ as defined above, which is $C^2$ on 
$\left[-l,l\right]\times\left[0,T\right)$ and
blows down uniformly, and let us
call it $z$. Using the same trick as above we can see
$z$ as a solution to the logarithmic diffusion equation on 
$\mathcal{M}$, and also such that $z\left(\cdot, 0\right)<u_0$.
The comparison principle applies again, and this time gives us
that $z< u$ on $\mathcal{M}\times \left[0,T\right)$, which then
shows that $u$ remains bounded away from 0 on $\left[0,T\right)$.
Hence we can conclude that the solution to (\ref{Ricciunnormalised2})
can be extended past $T$; therefore, the solution 
to (\ref{Ricciunnormalised2}) is global in time, which is what we wanted to show.

\section{The logarithmic diffusion equation on a disc}
\label{sect:diffusionondisc}
Consider the boundary value problem
\begin{equation}
\label{eq:logdiffusiondisc}
\left\{
    \begin{array}{l}
    \partial_t u =\Delta \log u \quad \mbox{in}\quad D_{a}\times\left(0,T\right),\\
    \dfrac{\partial u}{\partial \eta}=2\gamma u^{p} \quad
    \mbox{on} \quad \partial D_a\times\left(0,T\right),
    \end{array}
\right.
\end{equation}
where $D_a\subset \mathbb{R}^2$ is the disc of radius $a$ centered at the the origin,
and $\eta$ represents the outward unit normal at the boundary $\partial D_a$.
In this section we use the results in \cite{brendle02, cortissoz20} to prove the following.
\begin{theorem}
\label{thm:blowdownfinite2D}
Let $p\leq 1$ and $\gamma\leq -\dfrac{1}{2a}$.
There exists an $0<\epsilon\left(a\right)<1$ such that
if $u$ is a solution to (\ref{eq:logdiffusiondisc}),
with initial data such that $0<u_0<\epsilon$, then $u$ blows down
in finite time. 
\end{theorem}

The proof of this theorem follows from the following lemma, which is a
consequence of the results in \cite{brendle02,cortissoz20},
and the comparison principle.
\begin{lemma}
\label{lem:riccidisclemma}
Consider 
\begin{equation}
\label{eq:logdiffusiondiscb}
\left\{
    \begin{array}{l}
    \partial_t w =\Delta \log w \quad \mbox{in}\quad D_{a}\times\left(0,T\right),\\
    \dfrac{\partial w}{\partial \eta}=\left(2\beta\sqrt{w}-\frac{1}{a}\right)w \quad
    \mbox{on} \quad \partial D_a\times\left(0,T\right),
    \end{array}
\right.
\end{equation}
with $\beta\geq 0$.
Let $w_0>0$ be such that $\Delta \log w_0<0$. Assume that $u_0$ also satisfies the compatibility
condition
\[
\dfrac{\partial w_0}{\partial \eta}=\left(2\beta\sqrt{w_0}-\frac{1}{a}\right)w_0.
\]
Then the solution $w$ to (\ref{eq:logdiffusiondiscb}) with initial data $w_0$
blows-down in finite time, that is, $T<\infty$. Furthermore, we have the
estimate
\[
w_{\min}\left(t\right)\sim w_{\max}\left(t\right)\sim T-t.
\]
\end{lemma}

\begin{proof}
It is proved in \cite{cortissoz20} that the solution to the Ricci flow
on a disc with positive curvature and constant nonnegative geodesic curvature blows-up
in finite time. The following estimate for the minimum and the
maximum of the curvature is proved: 
 $\lim_{t\nearrow T} R_{\max}\left(t\right)/R_{\min}\left(t\right)\rightarrow 1$, and
 it is also shown that 
 \[
 R_{\max}\left(t\right)\sim \frac{1}{T-t}.
 \]
 Given $w$ a solution to (\ref{eq:logdiffusiondiscb}) as in the hypothesis of the lemma,
 then $g=wg_E$, where $g_E$ is the Euclidean metric on the disc $D_a$, is the
 solution to a Ricci flow on a disc with
 positive curvature and constant geodesic curvature equals to $\beta$.
 This shows that
 \[
 e^{-\int_0^t R_{\max}\left(\tau\right)\,d\tau}\leq
 w_{\min}\left(t\right)\leq w_{\max}\left(t\right)
 \leq
 e^{-\int_0^t R_{\min}\left(\tau\right)\,d\tau},
 \]
 and thus we have that
 \[
w_{\min}\left(t\right)\sim w_{\max}\left(t\right)\sim T-t,
\]
which is the conclusion of the lemma.
\end{proof}

\begin{proof}[Proof of Theorem \ref{thm:blowdownfinite2D}]
Take any solution to (\ref{eq:logdiffusiondiscb}) with $\beta=0$. By
Lemma \ref{lem:riccidisclemma}, there is $t_*$ such
that for $t>t_*$ we have that $\epsilon<w_{\min}<w_{\max}<1$. We may assume
without loss of generality that $t_*=0$. Since $\beta=0$, if $p\leq 1$,
at the boundary we have that 
\[
\frac{\partial w}{\partial \eta}=-\frac{1}{a} w\geq -\frac{1}{a} w^{p}\geq 2\gamma w^{p}.
\]
Then, by the
comparison principle $w>u$ as long as the solution $w$ exists. This implies
that the interval of existence of $u$ is contained in the maximal
interval of existence of $w$, which is finite; as $w$ blows-down so must $u$.
\end{proof}
The reader is invited to compare the previous theorem with the result proved
in Section \ref{subsect:finiteblowdown} on the finite time blow-down in the 
one dimensional case.

\end{document}